\documentclass[11 pt]{amsart}
\usepackage[latin1]{inputenc}
\usepackage{amsfonts,amssymb,amsmath,amsthm}
\usepackage{hyperref}
\usepackage{geometry}
\usepackage{graphicx, psfrag}
\usepackage{color}
\usepackage{setspace}

\newtheorem{theorem}{Theorem}
\newtheorem{acknow}{Acknowledgement}

\newtheorem{example}{Example}

\newtheorem{lemma}{Lemma}

\newtheorem{remark}{Remark}

\begin{document}
\title[Critical BPRE and Cauchy domain of attraction]{Critical branching
processes in random environment and Cauchy domain of attraction}
\author{C. Dong}
\address{Congzao Dong, School of Mathematics and Statistics, Xidian University, 710126 Xian, P.R. China}
\email{czdong@xidian.edu.cn}
\author{C. Smadi}
\address{Charline Smadi, Univ. Grenoble Alpes, IRSTEA, LESSEM, 38000
Grenoble, France and Univ. Grenoble Alpes, CNRS, Institut Fourier, 38000
Grenoble, France}
\email{charline.smadi@irstea.fr}
\author{V. A. Vatutin}
\address{Vladimir A. Vatutin, Department of Discrete Mathematics, Steklov Mathematical Institute of Russian Academy of Sciences, 8 Gubkin
Street, 117 966 Moscow GSP-1, Russia}
\email{vatutin@mi.ras.ru}

\maketitle

\begin{abstract}
We are interested in the survival probability of a population modeled by a
critical branching process in an i.i.d. random environment. We assume that
the random walk associated with the branching process is oscillating and
satisfies a Spitzer condition $\mathbf{P}(S_{n}>0)\rightarrow \rho ,\
n\rightarrow \infty $, which is a standard condition in fluctuation theory
of random walks. Unlike the previously studied case $\rho \in (0,1)$, we
investigate the case where the offspring distribution is in the domain of
attraction of a stable law with parameter $1$, which implies that $\rho =0$
or $1$. We find the asymptotic behaviour of the survival probability of the
population in these two cases.\\



\end{abstract}

\section{Introduction and main results}

Branching processes have been introduced by Galton and Watson in the 19th
century in order to study the extinction of family names \cite%
{watson1875probability}. Since then they have been widely used to model the
dynamics of populations or the spread of infections for instance \cite%
{haccou2005branching,allen2010introduction}. Branching processes in random
environment have been first introduced and studied by Smith and Wilkinson
and Athreya and Karlin in the early seventies \cite%
{smith1969branching,athreya1971branching,athreya19712branching}. By
introducing such processes, their aim was to better understand the effect of
the environmental stochasticity on the population dynamics. Initially
restricted to environments satisfying strong assumptions or to particular
offspring distributions, they have been later generalised. Their study has
known a renewed interest during the last two decades, with the development
of new techniques to investigate them, in particular by linking events on
the trajectory of the population process until a certain generation $n$ with
an other event of its associated random walk until the same time $n$ (see,
for instance, \cite%
{afanasyev2005criticality,afanasyev2012limit,bansaye2013lower,vatutin2013evolution}
for more detail).

A branching process in an independent identically distributed (i.i.d.)
random environment is specified by a sequence of i.i.d. random offspring
generating functions
\begin{equation*}
f_{n}(s):=\sum_{k=0}^{\infty }f_{n}\left[ k\right] s^{k},\quad n\in
\{1,2,...\}=:\mathbb{N},\quad 0\leq s\leq 1.  \label{DefF}
\end{equation*}%
Denoting by $Z_{n}$ the number of individuals in the process at time $n$, we
assume that there is initially one individual in the population ($Z_{0}=1$)
and we define its evolution by the relations
\begin{equation*}
\mathbf{E}[s^{Z_{n}}|f_{1},\ldots ,f_{n};Z_{0},Z_{1},\ldots
,Z_{n-1}]:=(f_{n}(s))^{Z_{n-1}},\quad n\in \mathbb{N}.
\end{equation*}

Let
\begin{equation*}
X_{k}:=\log f_{k}^{\prime }(1)=\log \mathbf{E}[Z_{k}|f_{k},Z_{k-1}=1],\quad
k\in \mathbb{N},
\end{equation*}%
and denote
\begin{equation*}
S_{0}:=0,\quad S_{n}:=X_{1}+X_{2}+\ldots +X_{n}
\end{equation*}%
the auxiliary random walk associated with the quenched expectation of
offspring number. The long time behaviour of the process $\mathcal{Z}%
:=\left\{ Z_{n},\,n\geq 0\right\} $ is intimately related to the properties
of the random walk $\mathcal{S}:=\left\{ S_{n},n\geq 0\right\} $ (see \cite%
{geiger2001survival,geiger2003limit,afanasyev2005criticality} for instance).
According to fluctuation theory of random walks (see \cite{fe}), three
different cases are possible: either $\mathcal{S}$ drifts to $\infty $, or $%
\mathcal{S}$ drifts to $-\infty $, or the random walk oscillates:
\begin{equation*}
\limsup_{n\rightarrow \infty }S_{n}=+\infty \ \text{ and }\
\liminf_{n\rightarrow \infty }S_{n}=-\infty
\end{equation*}%
with probability $1$. Accordingly, the branching process is called \textit{%
supercritical}, \textit{subcritical}, or \textit{critical} \cite{afanasyev2005criticality}. We consider the last possibility. In this case
the stopping time
\begin{equation*}
T^{-}:=\min \{k\geq 1:S_{k}<0\}
\end{equation*}%
is finite with probability $1$ and, as a result (see \cite%
{afanasyev2005criticality}), the extinction time
\begin{equation*}
T:=\min \{k\geq 1:Z_{k}=0\}
\end{equation*}%
of the process $\mathcal{Z}$ is finite with probability $1$.

In this work we will be interested in the asymptotic behaviour of the
survival probability $\mathbf{P}(Z_{n}>0)$ of the population at large time.
It is a natural question when dealing with populations, and it has been
answered under various assumptions in the case of branching processes in
random environment (see, for instance, \cite%
{kozlov1977asymptotic,geiger2001survival,afanasyev2005criticality,afanasyev2012limit}%
).

We assume that the random walk $\mathcal{S}$ satisfies the Doney-Spitzer
condition, which is a classical condition in fluctuation theory, and writes
\begin{equation}
\lim_{n\rightarrow \infty } \frac{1}{n} \sum_{m=1}^n \mathbf{P}\left(
S_{m}>0\right) =:\rho .  \label{Spit}
\end{equation}
According to Bertoin and Doney \cite{bertoin1997spitzer}, this condition is
equivalent to
\begin{equation*}
\lim_{n\rightarrow \infty }\mathbf{P}\left( S_{n}>0\right) =:\rho .
\end{equation*}
The case $\rho \in (0,1)$ has been studied by Afanasyev and his coauthors in
\cite{afanasyev2005criticality}. Under some mild additional assumptions they
proved the following equivalent for the survival probability of the
population at large times $n$,
\begin{equation}
\mathbf{P}(Z_{n}>0)\sim \frac{l(n)}{n^{1-\rho }},  \label{tails}
\end{equation}%
where $l(.)$ is a slowly varying function.

The aim of the present paper is to complement (\ref{tails}) by considering
the asymptotic behaviour of $\mathbf{P}(Z_{n}>0)$ as $n\rightarrow \infty $
in the cases $\rho =0$ and $\rho =1$.\newline

Before stating our main results, we need to introduce some notation and a
set of assumptions on the law of the random walk $\mathcal{S}$. The main
assumption is that $\mathcal{S}$ is in the domain of attraction of a stable
law with parameter $1$. It means that there exist a slowly varying function $%
L(\cdot )$, and two nonnegative numbers $p$ and $q,\ p+q=1,$ such that
\begin{equation}
\mathbf{P}\left( X_{1}>x\right) \sim p\frac{L(x)}{x}\quad \text{and}\quad
\mathbf{P}\left( X_{1}<-x\right) \sim q\frac{L(x)}{x},\quad x\rightarrow
\infty .  \label{BasicCond}
\end{equation}

As we will see (Remark \ref{rho0p>q}), $\mathcal{S}$ will satisfy the Doney-Spitzer
condition with $\rho =0$ (resp. $\rho =1$) in the case $p>q$ (resp. $p<q$).
To show that we introduce two scaling sequences which play the main role in
the asymptotic behaviour of various quantities related to the random walk $%
\mathcal{S}$. The first sequence, $\left\{ a_{n},n\in \mathbb{N}\right\} $,
satisfies, as $n\rightarrow \infty $ the relation
\begin{equation}
\frac{L(a_{n})}{a_{n}}\sim \frac{1}{n}.  \label{defan}
\end{equation}%
Note that the sequence is regularly varying with parameter 1 as $%
n\rightarrow \infty $ (see \cite{Sen76}). We can thus rewrite it as
\begin{equation}
a_{n}=nL_{4}(n)  \label{AsympAn}
\end{equation}%
where $L_{4}(.)$ is a slowly varying function as $n\rightarrow \infty $. The
second sequence, $\left\{ h_{n},n\in \mathbb{N}\right\} $, is specified by
\begin{equation}
h_{n}:=n\mu \left( a_{n}\right) \text{ where }\mu (x)=\mathbf{E}\left[ X_{1}%
\mathbf{1}_{\left\{ \left\vert X_{1}\right\vert \leq x\right\}} \right] ,
\label{defhn}
\end{equation}
where $\mathbf{1}$ is the indicator function.

In addition, we suppose that%
\begin{equation}
\mu :=\mathbf{E}\left[ X_{1}\right] =0.  \label{Def_mu}
\end{equation}

Let
\begin{equation}
l^{\ast }(z):=\int_{z}^{\infty }\frac{L\left( y\right) }{y}dy.
\label{deflstar}
\end{equation}%
The relation between $p$ and $q$ and the value of $\rho $ in \eqref{Spit}
derive from the following properties:

\begin{remark}
\label{rho0p>q} (Lemma 7.3 in \cite{Berg17} and Proposition 1.5.9 in \cite%
{bingham1989regular}) If the conditions (\ref{BasicCond}) and (\ref{Def_mu})
hold and $p>q$, then, as $n\rightarrow \infty $
\begin{equation*}
\mathbf{P}\left( S_{n}>0\right) \sim \frac{p}{p-q}\frac{L\left( \left\vert
h_{n}\right\vert \right) }{l^{\ast }\left( \left\vert h_{n}\right\vert
\right) }\quad  \label{AsymS_up}
\end{equation*}%
and
\begin{equation*}
\sum_{k=1}^{n}\frac{1}{k}\mathbf{P}\left( S_{k}>0\right) \sim -\frac{p}{p-q}%
\log l^{\ast }\left( \left\vert h_{n}\right\vert \right) .
\end{equation*}%
Notice that, as $n\rightarrow \infty $ , $l^{\ast }(n)\rightarrow 0$, $%
l^{\ast }(\cdot )$ is slowly varying and $l^{\ast }(n)/L(n)\rightarrow
\infty $. The case $p<q$ is symmetric.

Thus, the situation $p>q$ corresponds to $\rho =0$ while $p<q$ corresponds
to the case $\rho =1.$
\end{remark}

As in \cite{afanasyev2005criticality}, we need to impose restrictions on the
standardized truncated second moment of the environment, namely:
\begin{equation}
\zeta _{k}(a):=\sum_{y=a}^{\infty }y^{2}f_{k}[y]/\left( \sum_{y=0}^{\infty
}yf_{k}[y]\right) ^{2},  \label{defzetaa}
\end{equation}%
for $a,k\in \mathbb{N}$. The moment condition depends on the value of $\rho $
in the Doney-Spitzer condition \eqref{Spit}.\newline

\textbf{Condition A}. ($\rho =0\leftrightarrow p>q$) There exist $a\in
\mathbb{N}$ and $\beta >0$ such that
\begin{equation*}
\mathbf{E}[\zeta _{1}^{\beta }(a)]<\infty \quad \text{and}\quad \mathbf{E}%
[U(X_{1})\zeta _{1}^{\beta }(a)]<\infty ,
\end{equation*}%
where $U$ is the renewal function associated to the strict descending ladder
epochs of $\mathcal{S}$,
\begin{equation}
\gamma _{0}:=0,\quad \gamma _{j+1}:=\min \left( n>\gamma
_{j}:S_{n}<S_{\gamma _{j}}\right) ,\quad j\in \mathbb{N}_{0}:=\mathbb{N\cup }%
\left\{ 0\right\} ,  \label{defstrictdesc}
\end{equation}%
and is defined by
\begin{equation}
U(x):=\sum_{j=0}^{\infty }\mathbf{P}(S_{\gamma _{j}}\geq -x),\quad x>0,\quad
U\left( 0\right) =1,\quad U\left( x\right) =0,\quad x<0.  \label{Def_U}
\end{equation}

\textbf{Condition B}. ($\rho =1\leftrightarrow p<q$) There exist $a\in
\mathbb{N}$ and $\beta >0$ such that
\begin{equation*}
\mathbf{E}\left[ \left( \log ^{+}\zeta _{1}(a)\right) ^{1+\beta }\right]
<\infty \quad \text{and}\quad \mathbf{E}\left[ U(X_{1})\left( \log ^{+}\zeta
_{1}(a)\right) ^{1+\beta }\right] <\infty .
\end{equation*}

\vspace{.5cm}

Observe that the moment condition in \cite{afanasyev2005criticality} under
the Doney-Spitzer condition \eqref{Spit} with $\rho \in (0,1)$ was the
existence of $\beta >0$ and $a\in \mathbb{N}$ such that:
\begin{equation*}
\mathbf{E}\left[ \left( \log ^{+}\zeta _{1}(a)\right) ^{1/\rho +\beta }%
\right] <\infty \quad \text{and}\quad \mathbf{E}\left[ U(X_{1})\left( \log
^{+}\zeta _{1}(a)\right) ^{1+\beta }\right] <\infty .
\end{equation*}%
Our \textbf{Condition B} is thus a natural extension of the moment condition
to the case $\rho =1$. In contrast, such a natural extension for $\rho =0$
would have provided an infinite exponent for the logarithm and we could not
obtain a moment condition on the logarithm only. Notice however that we can
take $\beta $ as small as we want in \textbf{Condition A. }Thus\textbf{, }%
our moment condition is not very strong.\newline

Last, for technical reasons, we need to add an assumption which will be used
for the case $p>q$ only.\newline

\textbf{Condition C}. There exists an integer-valued
function $g(j)=e^{o(j)},\ j\rightarrow \infty ,$ such that
\begin{equation*}
\sum_{j=1}^{\infty }1/\Lambda(g(j))<\infty ,
\end{equation*}%
where $\Lambda$ is a slowly varying function (see the proof of Proposition 12 in \cite{kortchemski2019condensation}) defined by
\begin{equation}
 \label{def_Lambda}
 \Lambda \left(\frac{1}{1-s}\right)= \exp \left( \sum_{k=1}^\infty  \frac{\mathbf{P}(S_k \geq 0)}{k}s^k \right), \quad s \in [0,1).
\end{equation}

We will provide in Example \ref{ex1} an illustration of a slowly varying
function $L(\cdot )$ meeting this condition.\newline

As previously observed under different assumptions on the random environment
(see, for instance, \cite{vatutin2013evolution} for a comprehensive review
on the critical and subcritical cases (before 2013) or the recent monograph
\cite{GV2017}) the survival of a branching process in random environment is
essentially determined by its survival until the moment when the associated
random walk $\mathcal{S}$ attains its infimum. The idea is that if we divide
the trajectory of the process on the interval $[0,n]$ into two parts, one
before the running infimum of the random environment $\mathcal{S}$, and one
after this running infimum, the process will live in a favorable environment
after the running infimum of the random environment, and will thus survive
with a nonnegligible probability until time $n$, provided it survived until
the time of the running infimum. This is essentially, in words, the idea of
the proof of our main result (see Theorem \ref{T_main}). To state things
more rigorously, we introduce the running infimum of the random walk $%
\mathcal{S}$:
\begin{equation}
L_{n}:=\min \left\{ S_{0},S_{1},...,S_{n}\right\} ,\ n\in \mathbb{N}_{0}.
\label{defLn}
\end{equation}

Depending on the relative positions of $p$ and $q$ (defined in %
\eqref{BasicCond}) or equivalently on the value of $\rho $ ($0$ or $1$) we
have the two following possible asymptotics for the survival probability of
the process $\mathcal{Z}$:

\begin{theorem}
\label{T_main}Assume that Conditions (\ref{BasicCond}) and (\ref{Def_mu}) hold.

\begin{itemize}
\item If $p>q$, and Conditions A and C hold then there exists a constant $%
K_{1}\in \left( 0,\infty \right) $ such that, as $n\rightarrow \infty $%
\begin{equation}
\mathbf{P}\left( Z_{n}>0\right) \sim K_{1}\mathbf{P}\left( L_{n}\geq
0\right) \sim K_{1}\frac{L_{22}(n)}{n},  \label{survivP>Q}
\end{equation}%
where $L_{22}(.)$ is a function slowly varying at infinity.

\item If $p<q$ and Condition B holds then there exists a constant $K_{2}\in
\left( 0,\infty \right) $ such that, as $n\rightarrow \infty $%
\begin{equation*}
\mathbf{P}\left( Z_{n}>0\right) \sim K_{2}\mathbf{P}\left( L_{n}\geq
0\right) \sim K_{2}L_{33}(n),  \label{survivP<Q}
\end{equation*}%
where $L_{33}(.)$ is a function slowly varying at infinity.
\end{itemize}
\end{theorem}

Hence, despite the irregular behaviour of the associated random walk $S$ (a
null expectation but a probability converging to $1$ to be positive (resp.
negative)), the asymptotic behaviour of the survival probability is, except
for the slowly varying function, the limit of the one obtained in \cite%
{afanasyev2005criticality} by taking $\rho=0$ or $1$ instead of $\rho \in
(0,1)$.\newline

The rest of the paper is structured as follows. Section \ref{sec_supr} is
dedicated to the study of the running extrema of the random walk $S$. In
Section \ref{sec_chgtprob}, we perform a change of measure, obtained as a
Doob-h transform, where the renewal function $U(\cdot )$ of $\mathcal{S}$
and the indicator of the event $\{L_{n}\geq 0\}$ are involved. Finally, the
proof of the main result, Theorem \ref{T_main}, is completed in Section \ref%
{sec_proof}.

\section{Estimates for the suprema of the associated random walk}

\label{sec_supr}

The aim of this section is to provide some bounds for the probabilities of
the events related to the running infimum and maximum of the random walk $%
\mathcal{S}$. We recall the definition of the running infimum in %
\eqref{defLn}, and introduce the running maximum via
\begin{equation*}
M_{n}:=\max \left\{ S_{1},...,S_{n}\right\} ,\ n\in \mathbb{N}.
\end{equation*}

We first list a number of known results which will be needed in our
arguments. Recall definitions \eqref{defan} and \eqref{defhn}. The following
results have been first derived in Theorem 3.4 in \cite{Berg17}, and then under weaker
conditions in \cite{kortchemski2019condensation} (see Proposition 12 and Remark 13).

\begin{theorem}
\label{T_Berg3_4} Assume that Conditions (\ref{BasicCond}) and (\ref{Def_mu})
hold. Then when $n$ goes to infinity,

1) if $p>q$ then $h_{n}\sim -\left( p-q\right) nl^{\ast }\left( a_{n}\right)
\rightarrow -\infty $ and
\begin{equation}
\mathbf{P}\left( L_{n}\geq 0\right) \sim \frac{L\left( \left\vert
h_{n}\right\vert \right) }{\left\vert
h_{n}\right\vert}\Lambda(n) =:\frac{L_{22}(n)}{n},  \label{Ber1}
\end{equation}
for some slowly varying functions $L_{22}$ (recall that the sowly varying function $\Lambda$ has been defined in \eqref{def_Lambda}).

2) if $p<q$ then $h_{n}\sim \left( q-p\right) nl^{\ast }\left( a_{n}\right)
\rightarrow +\infty $ and
\begin{equation}
\mathbf{P}\left( L_{n}\geq 0\right) \sim \frac{1}{\tilde{\Lambda}(n)}=:L_{33}(n),
\label{Ber2}
\end{equation}
for some slowly varying functions $L_{33}$. The slowly varying function $\tilde{\Lambda}$ is defined as $\Lambda$ but with $-\mathcal{S}$ in place of $\mathcal{S}$.

3) if $p>q$ then%
\begin{equation}
\mathbf{P}\left( M_{n}<0\right) \sim \frac{1}{\Lambda(n)} =:L_{44}(n)
\label{Ber3}
\end{equation}

4) if $p<q$ then
\begin{equation}
\mathbf{P}\left( M_{n}<0\right) \sim\frac{L\left( \left\vert
h_{n}\right\vert \right) }{\left\vert
h_{n}\right\vert}\tilde{\Lambda}(n)  =:\frac{L_{55}(n)}{n}.
\label{Ber4}
\end{equation}%
\end{theorem}

Recall the definitions of the strict descending ladder epochs $\left\{
\gamma _{j},j\in \mathbb{N}_{0}\right\} $ of $\mathcal{S}$ and of their
associated renewal function $U(\cdot )$ in \eqref{defstrictdesc} and (\ref%
{Def_U}), respectively, and introduce the strict ascending ladder epochs $%
\left\{ \Gamma _{j},j\in \mathbb{N}_{0}\right\} $ of $\mathcal{S}$ and their
associated renewal function $V(\cdot )$ via
\begin{equation*}
\Gamma _{0}:=0,\quad \Gamma _{j+1}:=\min (n>\Gamma _{j}:S_{n}>S_{\Gamma
_{j}}),\quad j\in \mathbb{N}_{0},  \label{defstrictasc}
\end{equation*}%
and
\begin{equation*}
V(x):=1+\sum_{j=1}^{\infty }\mathbf{P}(S_{\Gamma _{j}}<x),\quad x>0,\quad
V\left( 0\right) =1,\quad V(x)=0,\quad x<0.
\end{equation*}

For a slowly varying function $L_{ii}(\cdot )$ let
\begin{equation*}
\hat{l}_{ii}(n):=\int_{1}^{n}\frac{L_{ii}(x)}{x}dx.
\end{equation*}%
The next lemma provides bounds on the probabilities for the running extrema
to be in a certain interval.

\begin{lemma}
\label{L_roughEstimates}Assume that Conditions (\ref{BasicCond}) and (\ref%
{Def_mu}) hold. Then there exists a constant $C\in \left( 0,\infty
\right) $ such that, for every $x\geq 0$ and $n\in \mathbb{N}$,
\begin{equation}
\mathbf{P}\left( L_{n}\geq -x\right) \leq \left\{
\begin{array}{ccc}
CU(x)n^{-1}\hat{l}_{22}\left( n\right) & \text{if} & p>q, \\
&  &  \\
CU(x)L_{33}\left( n\right) & \text{if} & p<q,%
\end{array}%
\right.  \label{EstimMin2}
\end{equation}%
and
\begin{equation*}
\mathbf{P}\left( M_{n}<x\right) \leq \left\{
\begin{array}{ccc}
CV(x)L_{44}\left( n\right) & \text{if} & p>q, \\
&  &  \\
CV(x)n^{-1}\hat{l}_{55}\left( n\right) & \text{if} & p<q.%
\end{array}%
\right.  \label{EstimMax2}
\end{equation*}
\end{lemma}

\begin{proof}
We know by a Spitzer identity that, for any $\lambda \geq 0$%
\begin{eqnarray*}
\sum_{n=0}^{\infty }s^{n}\mathbf{E}\left[ e^{\lambda L_{n}}\right]  &=&\exp
\left\{ \sum_{n=1}^{\infty }\frac{s^{n}}{n}\mathbf{E}\left[ e^{\lambda \min
(0,S_{n})}\right] \right\}  \\
&=&\exp \left\{ \sum_{n=1}^{\infty }\frac{s^{n}}{n}\mathbf{E}\left[
e^{\lambda S_{n}};S_{n}<0\right] \right\} \exp \left\{ \sum_{n=1}^{\infty }%
\frac{s^{n}}{n}\mathbf{P}\left( S_{n}\geq 0\right) \right\} .
\end{eqnarray*}%
A Sparre-Anderson identity (see, for instance, Theorem 4.3 in \cite{GV2017})
allows us to rewrite the first term at the right hand side as
\begin{eqnarray*}
\exp \left\{ \sum_{n=1}^{\infty }\frac{s^{n}}{n}\mathbf{E}\left[ e^{\lambda
S_{n}};S_{n}<0\right] \right\}  &=&1+\sum_{n=1}^{\infty }s^{n}\mathbf{E}%
\left[ e^{\lambda S_{n}};\Gamma ^{\prime }>n\right]  \\
&=&\int_{0}^{+\infty }e^{-\lambda x}U_{s}(dx),
\end{eqnarray*}%
where
\begin{equation*}
\Gamma ^{\prime }:=\min \left( n\in \mathbb{N},S_{n}\geq 0\right)
\end{equation*}%
and
\begin{equation*}
U_{s}(x)=\sum_{n=0}^{\infty }s^{n}\mathbf{P}\left( S_{n}\geq -x;\Gamma
^{\prime }>n\right) ,\ x\geq 0.
\end{equation*}%
Therefore,
\begin{eqnarray}
\sum_{n=0}^{\infty }s^{n}\mathbf{P}\left( L_{n}\geq -x\right)
&=&U_{s}(x)\exp \left\{ \sum_{n=1}^{\infty }\frac{s^{n}}{n}\mathbf{P}\left(
S_{n}\geq 0\right) \right\}   \notag \\
&=&U_{s}(x)\sum_{n=1}^{\infty }s^{n}\mathbf{P}\left( L_{n}\geq 0\right)
\label{Ln0}
\end{eqnarray}%
for $x\geq 0$. Note that by the duality principle for random walks (see, for
instance, \cite{GV2017} p. 63),
\begin{eqnarray}
\lim_{s\uparrow 1}U_{s}(x) &=&\sum_{n=0}^{\infty }\mathbf{P}\left( S_{n}\geq
-x;\Gamma ^{\prime }>n\right)   \notag  \label{duality} \\
&=&1+\sum_{n=1}^{\infty }\mathbf{P}\left( S_{n}\geq
-x;S_{i}<0,i=1,...,n\right)   \notag \\
&=&1+\sum_{n=1}^{\infty }\mathbf{P}\left( S_{n}\geq
-x;S_{n}<S_{j},j=0,1,...,n-1\right)   \notag \\
&=&1+\sum_{n=1}^{\infty }\sum_{r=1}^{n}\mathbf{P}\left( S_{n}\geq -x;\gamma
_{r}=n\right)  \\
&=&1+\sum_{r=1}^{\infty }\sum_{n=r}^{\infty }\mathbf{P}\left( S_{n}\geq
-x;\gamma _{r}=n\right) =1+\sum_{r=1}^{\infty }\mathbf{P}\left( S_{\gamma
_{r}}\geq -x\right) =U(x).  \notag
\end{eqnarray}

On the other hand, if $s\uparrow 1$ then (\ref{Ber1}) and an application of
Corollary 1.7.3 in \cite{bingham1989regular} with $\rho =0$ give for $p>q$,
\begin{equation*}
\sum_{n=0}^{\infty }s^{n}\mathbf{P}\left( L_{n}\geq 0\right) \sim
\sum_{n=0}^{\infty }s^{n}\frac{L_{22}(n)}{n}\sim \hat{l}_{22}\left( \frac{1}{%
1-s}\right) ,
\end{equation*}%
while (\ref{Ber2}) and again an application of Corollary 1.7.3 in \cite%
{bingham1989regular} but now with $\rho =1$ justify, for $p<q$ the
asymptotics
\begin{equation*}
\sum_{n=0}^{\infty }s^{n}\mathbf{P}\left( L_{n}\geq 0\right) \sim
\sum_{n=0}^{\infty }s^{n}L_{33}(n)\sim \frac{L_{33}(1/\left( 1-s\right) )}{%
1-s}.
\end{equation*}

Thus if $p>q$ then, as $s\uparrow 1$%
\begin{equation}
\sum_{n=0}^{\infty }s^{n}\mathbf{P}\left( L_{n}\geq -x\right) \sim U(x)\hat{l%
}_{22}\left( \frac{1}{1-s}\right) ,  \label{Slowly1}
\end{equation}%
and if $p<q$ then, as $s\uparrow 1$%
\begin{equation*}
\sum_{n=0}^{\infty }s^{n}\mathbf{P}\left( L_{n}\geq -x\right) \sim U(x)\frac{%
L_{33}(1/\left( 1-s\right) )}{1-s}.
\end{equation*}

Using \eqref{Ln0}, \eqref{Slowly1} and the monotonicity of $U_{s}(x)$ in $s$
we get, for $p>q$%
\begin{equation*}
\sum_{n=0}^{\infty }s^{n}\mathbf{P}\left( L_{n}\geq -x\right) \leq
U(x)\sum_{n=0}^{\infty }s^{n}\mathbf{P}\left( L_{n}\geq 0\right) \sim U(x)%
\hat{l}_{22}\left( \frac{1}{1-s}\right) .
\end{equation*}

Since $\mathbf{P}\left( L_{n}\geq -x\right) $ is nonincreasing with $n,$ we
have for $p>q$%
\begin{eqnarray*}
\frac{n}{2}\left( 1-\frac{1}{n}\right) ^{n}\mathbf{P}\left( L_{n}\geq
-x\right) &\leq &\sum_{n/2\leq m\leq n}\left( 1-\frac{1}{n}\right) ^{m}%
\mathbf{P}\left( L_{m}\geq -x\right) \\
&\leq &CU(x)\hat{l}_{22}\left( n\right) ,
\end{eqnarray*}%
and, similarly, for $p<q$%
\begin{equation*}
\frac{n}{2}\left( 1-\frac{1}{n}\right) ^{n}\mathbf{P}\left( L_{n}\geq
-x\right) \leq CU(x)nL_{33}\left( n\right) .
\end{equation*}%
As a result%
\begin{equation*}
\mathbf{P}\left( L_{n}\geq -x\right) \leq \left\{
\begin{array}{ccc}
CU(x)n^{-1}\hat{l}_{22}\left( n\right) & \text{if} & p>q, \\
&  &  \\
CU(x)L_{33}\left( n\right) & \text{if} & p<q.%
\end{array}%
\right.
\end{equation*}%
By the same arguments and (\ref{Ber3}) we have as $s\uparrow 1$
\begin{equation*}
\sum_{n=1}^{\infty }s^{n}\mathbf{P}\left( M_{n}<0\right) \sim
\sum_{n=1}^{\infty }s^{n}L_{44}(n)\sim \frac{L_{44}(1/\left( 1-s\right) )}{%
1-s}  \label{Regularity2}
\end{equation*}%
for $p>q$, and by (\ref{Ber4}) \
\begin{equation*}
\sum_{n=1}^{\infty }s^{n}\mathbf{P}\left( M_{n}<0\right) \sim
\sum_{n=1}^{\infty }s^{n}\frac{L_{55}(n)}{n}\sim \hat{l}_{55}\left( \frac{1}{%
1-s}\right)  \label{Slowly2}
\end{equation*}%
for $p<q$. Thus
\begin{equation*}
\mathbf{P}\left( M_{n}<x\right) \leq \left\{
\begin{array}{ccc}
CV(x)L_{44}\left( n\right) & \text{if} & p>q, \\
&  &  \\
CU(x)n^{-1}\hat{l}_{55}\left( n\right) & \text{if} & p<q.%
\end{array}%
\right.
\end{equation*}%
This ends the proof.
\end{proof}

\begin{remark}
Observe that%
\begin{multline*}
\sum_{n=1}^{\infty }s^{n}\mathbf{P}\left( L_{n}\geq 0\right)
\sum_{n=1}^{\infty }s^{n}\mathbf{P}\left( M_{n}<0\right) \\
=\exp \left\{ \sum_{n=1}^{\infty }\frac{s^{n}}{n}\mathbf{P}\left( S_{n}\geq
0\right) \right\} \times \exp \left\{ \sum_{n=1}^{\infty }\frac{s^{n}}{n}%
\mathbf{P}\left( S_{n}<0\right) \right\} =\frac{1}{1-s}.
\end{multline*}%
Thus, as $n\rightarrow \infty $
\begin{equation}
\hat{l}_{22}\left( n\right) L_{44}(n)\sim 1,\;L_{33}(n)\hat{l}_{55}\left(
n\right) \sim 1.  \label{Equival}
\end{equation}
\end{remark}

Set
\begin{equation*}
b_{n}:=\left( na_{n}\right) ^{-1},\quad n\in \mathbb{N}.
\end{equation*}%
The next statement describes some properties of the running extrema of $%
\mathcal{S}$.

\begin{lemma}
\label{P_AGKV}(compare with Proposition 2.3 in \cite{afanasyev2012limit})
Assume that Conditions (\ref{BasicCond}) and (\ref{Def_mu}) hold.
Then there exists a constant $c$ such that, uniformly for all $x,y\geq 0$
and all $n\in \mathbb{N}$
\begin{equation}
\mathbf{P}_{x}\left( L_{n}\geq 0,y-1\leq S_{n}<y\right) \ \leq \
c\,b_{n}\,U(x)V(y)\ ,  \label{LocalMin}
\end{equation}%
and
\begin{equation*}
\mathbf{P}_{-x}\left( M_{n}<0,-y\leq S_{n}<-y+1\right) \ \leq \
c\,b_{n}\,V(x)U(y)\ .  \label{LocalMax}
\end{equation*}
\end{lemma}

\begin{proof}
We prove the latter statement only. Since the density of any $\alpha $%
-stable law is bounded, it follows from the Gnedenko \cite{GK54} and Stone
\cite{Sto65} local limit theorems that there exists a finite constant $C$
such that for all $n\in \mathbb{N}$ and all $z,\Delta \geq 0,$
\begin{equation}
\mathbf{P}\left( S_{n}\in \lbrack -z,-z+\Delta )\right) \leq \frac{C\Delta }{%
a_{n}}.  \label{EstS1}
\end{equation}%
Let $x,y\geq 0$, $\mathcal{S}^{\prime }$ be the dual random walk
\begin{equation*}
S_{i}^{\prime }=S_{n}-S_{n-i}
\end{equation*}%
and $L_{i}^{\prime }$, $i\leq n$, the corresponding minima. Denote
\begin{align*}
A_{n}\ :=\ \{& M_{\lfloor n/3\rfloor }<x\} \\
A_{n}^{\prime }\ :=\ \{& L_{\lfloor n/3\rfloor }^{\prime }\geq -y\}\ , \\
A_{n}^{\prime \prime }\ :=\ \{& x-y\leq S_{n}<x-y+1\} \\
=\ \{& x-y-T_{n}\leq S_{\lfloor 2n/3\rfloor }-S_{\lfloor n/3\rfloor
}<x-y-T_{n}+1\}\ ,
\end{align*}%
with
\begin{equation*}
T_{n}:=S_{\lfloor n/3\rfloor }+S_{n}-S_{\lfloor 2n/3\rfloor }.
\end{equation*}%
Let $\mathcal{A}_{n}$ be the $\sigma $--field generated by $X_{1},\ldots
,X_{\lfloor n/3\rfloor }$ and $X_{\lfloor 2n/3\rfloor +1},\ldots ,X_{n}$.
Then $T_{n}$ is $\mathcal{A}_{n}$--measurable, whereas $S_{\lfloor
2n/3\rfloor }-S_{\lfloor n/3\rfloor }$ is independent of $\mathcal{A}_{n}$.
Consequently from (\ref{EstS1}) and the fact that $\left\{ a_{n},n\in
\mathbb{N}\right\} $ is regularly varying there is a $c>0$ such that
\begin{equation*}
\mathbf{P}\left( A_{n}^{\prime \prime }\,|\,\mathcal{A}_{n}\right) \ \leq \
ca_{n}^{-1}\ .
\end{equation*}%
Since $A_{n},A_{n}^{\prime }$ are $\mathcal{A}_{n}$-measurable and
independent, it follows that
\begin{equation*}
\mathbf{P}\left( A_{n}\cap A_{n}^{\prime }\cap A_{n}^{\prime \prime }\right)
\ \leq \ ca_{n}^{-1}\mathbf{P}\left( A_{n}\right) \mathbf{P}\left(
A_{n}^{\prime }\right) \ .
\end{equation*}%
Moreover, according to Lemma \ref{L_roughEstimates}
\begin{equation*}
\mathbf{P}\left( L_{\lfloor n/3\rfloor }^{\prime }\geq -y\right) \ \leq \
c_{1}U(y)n^{-1}\hat{l}_{22}\left( n\right) \ ,\quad \mathbf{P}(M_{\lfloor
n/3\rfloor }<x)\ \leq \ c_{2}V(x)L_{44}(n),
\end{equation*}%
if $p>q$ and
\begin{equation*}
\mathbf{P}\left( L_{\lfloor n/3\rfloor }^{\prime }\geq -y\right) \ \leq \
c_{1}U(y)L_{33}\left( n\right) \ ,\quad \mathbf{P}(M_{\lfloor n/3\rfloor
}<x)\ \leq \ c_{2}V(x)n^{-1}\hat{l}_{55}(n),
\end{equation*}%
if $p<q$. This and (\ref{Equival}) give the uniform estimate
\begin{equation*}
\mathbf{P}\left( A_{n}\cap A_{n}^{\prime }\cap A_{n}^{\prime \prime }\right)
\ \leq \ cV(x)U(y)\,b_{n}
\end{equation*}%
for $c$ sufficiently large. Now notice that
\begin{equation*}
\{M_{n}<x,x-y\leq S_{n}<x-y+1\}\ \subset \ A_{n}\cap A_{n}^{\prime }\cap
A_{n}^{\prime \prime }\ .
\end{equation*}%
The fact that the event on the left hand side is included in $A_{n}\cap
A_{n}^{\prime \prime }$ is straightforward. It is also included in $%
A_{n}^{\prime }$ due to the following series of inequalities, which hold for
any $0\leq i\leq n$ on the event $\{x-y\leq S_n,M_{n}<x\}$:
\begin{equation*}
x-y\leq S_{n}-S_{i}+S_{i}\leq S_{n}-S_{i}+M_{n}\leq
S_{n}-S_{i}+x=S_{n-i}^{\prime }+x.
\end{equation*}%
This ends the proof.
\end{proof}

We have now all the tools needed to prove the following statement.

\begin{lemma}
\label{L_minim} Assume that Conditions (\ref{BasicCond}) and (\ref{Def_mu}) hold. Then for every $x\geq 0$ as $n\rightarrow \infty $

1) if $p>q$ then
\begin{eqnarray}
\mathbf{P}\left( L_{n}\geq -x\right) &\sim &U(x)\mathbf{P}\left( L_{n}\geq
0\right) \sim U(x)\frac{L_{22}(n)}{n},\quad  \label{ExectMin3} \\
\mathbf{P}\left( M_{n}<x\right) &\sim &V(x)\mathbf{P}\left( M_{n}<0\right)
\sim V(x)L_{44}(n);\quad  \notag
\end{eqnarray}

2) if $p<q$ then%
\begin{eqnarray}
\mathbf{P}\left( L_{n}\geq -x\right) &\sim &U(x)\mathbf{P}\left( L_{n}\geq
0\right) \sim U(x)L_{33}(n),\quad  \label{ExectMin4} \\
\mathbf{P}\left( M_{n}<x\right) &\sim &V(x)\mathbf{P}\left( M_{n}<0\right)
\sim V(x)\frac{L_{55}(n)}{n}.  \notag
\end{eqnarray}
\end{lemma}

\begin{proof}
As the derivations of the four equivalents are similar, we only check the
first one. Let
\begin{equation*}
\tau _{n}:=\min \left\{ j\leq n:S_{j}=L_{n}\right\} .
\end{equation*}%
We have
\begin{eqnarray*}
\mathbf{P}\left( L_{n}\geq -x\right)  &=&\sum_{j=0}^{n}\mathbf{P}\left(
L_{n-j}\geq 0\right) \mathbf{P}\left( S_{j}\geq -x;\tau _{j}=j\right)  \\
&=&\sum_{j=0}^{n}\mathbf{P}\left( L_{n-j}\geq 0\right) \mathbf{P}\left(
S_{j}\geq -x;M_{j}<0\right) ,
\end{eqnarray*}%
where we used the duality principle as in \eqref{duality}. In view of (\ref%
{Ber1}), for any $\varepsilon \in \left( 0,1\right) $ and $j\leq \varepsilon
n$,
\begin{equation*}
\mathbf{P}\left( L_{n-j}\geq 0\right) \sim \frac{n}{n-j}\mathbf{P}\left(
L_{n}\geq 0\right) ,\quad n\rightarrow \infty .
\end{equation*}%
Moreover, from \eqref{duality}, we have
\begin{equation}
\sum_{j=0}^{n\varepsilon }\mathbf{P}\left( S_{j}\geq -x;M_{j}<0\right)
=\sum_{j=0}^{n\varepsilon }\mathbf{P}\left( S_{j}\geq -x;\Gamma ^{\prime
}>j\right) \sim U(x),\quad n\rightarrow \infty .  \label{Main_term}
\end{equation}%
We deduce that
\begin{equation*}
\sum_{j=0}^{n\varepsilon }\mathbf{P}\left( L_{n-j}\geq 0\right) \mathbf{P}%
\left( S_{j}\geq -x;M_{j}<0\right) -\mathbf{P}\left( L_{n}\geq 0\right)
U(x)=O(\varepsilon )\mathbf{P}\left( L_{n}\geq 0\right) U(x)
\end{equation*}%
when $n$ is large enough. Further, by (\ref{LocalMin}), (\ref{Ber1}) and %
\eqref{AsympAn}, for any $\delta >0$
\begin{eqnarray}
\sum_{j=n\varepsilon }^{n}\mathbf{P}\left( L_{n-j}\geq 0\right) \mathbf{P}%
\left( S_{j}\geq -x;M_{j}<0\right)  &\leq &Cb_{n}xU(x)\sum_{j=n\varepsilon
}^{n}\mathbf{P}\left( L_{n-j}\geq 0\right)   \notag \\
&\leq &\frac{CxU(x)}{na_{n}}L_{22}(n)  \notag \\
&=&\frac{CxU(x)}{n^{2}}\frac{L_{22}(n)}{L_{4}(n)}=o\left( \frac{1}{%
n^{2-\delta }}\right)   \notag \\
&=&o\left( \mathbf{P}\left( L_{n}\geq 0\right) \right) ,\quad n\rightarrow
\infty ,  \label{Neglterm}
\end{eqnarray}%
since $(L_{22}(n)/L_{4}(n))n^{-\delta }\rightarrow 0$ as $n\rightarrow
\infty $ for any $\delta >0.$ Combining (\ref{Main_term}) and (\ref{Neglterm}%
) and letting $\varepsilon \rightarrow 0$ give (\ref{ExectMin3}).
\end{proof}

The last result of this section is a technical statement which will be
needed in the proof of Theorem \ref{T_main}. As Lemma \ref{L_minim}, it is a
consequence of Lemma \ref{P_AGKV} and can be proven in the same way as
Corollary 2.4 in \cite{afanasyev2012limit}.

\begin{lemma}
\label{C_exponential} Assume that Conditions (\ref{BasicCond}) and (\ref{Def_mu}%
) hold. For any $\theta >0$ there exists a finite $c$ (depending
on $\theta $) such that for all $x,y\geq 0$
\begin{equation*}
\mathbf{E}_{x}\big[e^{-\theta S_{n}};L_{n}\geq 0,S_{n}\geq y\big ]\leq c\
b_{n}V(x)U(y)\ e^{-\theta y}
\end{equation*}%
and
\begin{equation*}
\mathbf{E}_{-x}\big[e^{\theta S_{n}};M_{n}<0,S_{n}<-y\big ]\leq c\
b_{n}V(y)U(x)\ e^{\theta y}\ .
\end{equation*}
\end{lemma}

\section{Change of measure}

\label{sec_chgtprob}

Recall the definition of the renewal function $U$ in (\ref{Def_U}). One of
its fundamental properties is the identity (see, for instance, \cite%
{kozlov1977asymptotic,bertoin1994conditioning})
\begin{equation*}
\mathbf{E}\left[ U(x+X);X+x\geq 0\right] =U(x),\,x\geq 0.  \label{DefV}
\end{equation*}%
This property has often been used to construct a change of probability
measure (see for instance \cite{geiger2001survival}), and we will use such a
construction in our proof.

Denote by ${\mathcal{F}}$ the filtration consisting of the $\sigma -$%
algebras ${\mathcal{F}}_{n}$ generated by the random variables $%
S_{0},...,S_{n}$ and $Z_{0},...,Z_{n}$. \ Taking into account $U(0)=1$ we
may introduce probability measures $\mathbf{P}_{n}^{+}$ on the $\sigma $%
-fields $\mathcal{F}_{n}$ by means of the densities
\begin{equation*}
d\mathbf{P}_{n}^{+}\ :=\ U(S_{n})I_{\{L_{n}\geq 0\}}\,d\mathbf{P}\ .
\end{equation*}%
Because of the martingale property the measures are consistent, i.e., $%
\mathbf{P}_{n+1}^{+}|\mathcal{F}_{n}=\mathbf{P}_{n}^{+}$. Therefore
(choosing a suitable underlying probability space), there exists a
probability measure $\mathbf{P}^{+}$ on the $\sigma $-field $\mathcal{F}%
_{\infty }:=\bigvee_{n}\mathcal{F}_{n}$ such that
\begin{equation}
\mathbf{P}^{+}|\mathcal{F}_{n}\ =\ \mathbf{P}_{n}^{+}\,,\quad n\geq 0\,.
\label{ppp}
\end{equation}%
We note that (\ref{ppp}) can be rewritten as
\begin{equation}
\mathbf{E}^{+}\,\left[ Y_{n}\right] \ =\ \mathbf{E}[Y_{n}U(S_{n});L_{n}\geq
0]  \label{measurechange}
\end{equation}%
for every $\mathcal{F}_{n}$--measurable nonnegative random variable $Y_{n}$.
This change of measure is the well-known Doob $h$-transform from the theory
of Markov processes. In particular, under $\mathbf{P}^{+}$ the process $S$
becomes a Markov chain with state space $\mathbb{R}_{0}^{+}$ and transition
kernel
\begin{equation*}
P^{+}(x;dy)\ :=\ \frac{1}{U(x)}\mathbf{P}\left( x+X\in dy\right) U(y)\
,\quad x\geq 0\ .
\end{equation*}%
In our context, we can show that $\mathbf{P}^{+}$ can be realised as the
limit of the probability of the process conditioned to live in a nonnegative
environment (in the sense that the running infimum is null). It is the
content of the next lemma, and will allow us to link the survival
probability of the population process to the probability for the running
infimum to be null, in order to prove Theorem \ref{T_main}.

\begin{lemma}
\label{conditioning} (compare with Lemma 2.5 in \cite%
{afanasyev2005criticality}) Assume that Conditions (\ref{BasicCond}) and (\ref%
{Def_mu}) hold. For $k\in \mathbb{N}$ let $Y_{k}$ be a bounded
real-valued $\mathcal{F}_{k}$--measurable random variable. Then, as $%
n\rightarrow \infty $,
\begin{equation*}
\mathbf{E}[Y_{k}\;|\;L_{n}\geq 0]\ \rightarrow \ \mathbf{E^{+}}\,Y_{k}\ .
\label{claim1_L_conditioning}
\end{equation*}%
More generally, let $Y_{1},Y_{2},\ldots $ be a uniformly bounded sequence of
real-valued random variables adapted to the filtration $\mathcal{F}$, which
converges $\mathbf{P}^{+}$--$a.s.$ to some random variable $Y_{\infty }$.
Then, as $n\rightarrow \infty $,
\begin{equation*}
\mathbf{E}[Y_{n}\;|\;L_{n}\geq 0]\ \rightarrow \ \mathbf{E}^{+}\;Y_{\infty
}\ .
\end{equation*}
\end{lemma}

\begin{proof}
The proof of this lemma in the case $p>q$ coincides with the proof of Lemma
2.5 in \cite{afanasyev2005criticality} when taking $\rho =0$ and we omit it.
In the case $p<q$ some modifications are needed to check the second claim of
the lemma. Namely, writing
\begin{equation*}
m_{l}(x):=\mathbf{P}(L_{l}\geq -x)\quad \text{for}\quad x\geq 0,\ l\in
\mathbb{N}
\end{equation*}%
and using \eqref{EstimMin2}, \eqref{ExectMin4} and \eqref{measurechange} we
deduce for $\lambda >1$, $k\leq n$ and $n$ large enough, the existence of a
finite $C$ such that
\begin{align*}
\left\vert \mathbf{E}[Y_{n}-Y_{k}|I\left\{ L_{\lfloor \lambda n\rfloor }\geq
0\right\} ]\right\vert & \leq \mathbf{E}\left[ \left\vert
Y_{n}-Y_{k}\right\vert \frac{m_{\lfloor (\lambda -1)n\rfloor }(S_{n})}{%
m_{\lfloor \lambda n\rfloor }(0)}I\left\{ L_{n}\geq 0\right\} \right]  \\
& \leq C\mathbf{E}\left[ \left\vert Y_{n}-Y_{k}\right\vert U(S_{n})I\left\{
L_{n}\geq 0\right\} \right]  \\
& =C\mathbf{E}^{+}\left[ \left\vert Y_{n}-Y_{k}\right\vert \right] .
\end{align*}%
Letting\ sequentially $n$ and $k$ go to infinity and applying the dominated
convergence theorem, we obtain that the right hand side of the previous
series of inequalities vanishes. Applying now the first claim of the lemma
and using the fact that $n\mapsto \mathbf{P}(L_{n}\geq 0)$ is slowly varying
we obtain
\begin{equation*}
\mathbf{E}[Y_{n};L_{\lfloor \lambda n\rfloor }\geq 0]=\left( \mathbf{E}%
^{+}[Y_{\infty }]+o(1)\right) \mathbf{P}(L_{\lfloor \lambda n\rfloor }\geq
0)=\left( \mathbf{E}^{+}[Y_{\infty }]+o(1)\right) \mathbf{P}(L_{n}\geq 0)
\end{equation*}%
and
\begin{equation*}
\mathbf{E}[Y_{n};L_{n}\geq 0]-\mathbf{E}[Y_{n};L_{\lfloor \lambda n\rfloor
}\geq 0]=o\left( \mathbf{P}(L_{n}\geq 0)\right) .
\end{equation*}%
This ends the proof.
\end{proof}

Let $\nu \geq 1$ be the time of the first \emph{prospective minimal value}
of $\mathcal{S}$, i.e., a minimal value with respect to the future development of the
walk,
\begin{equation*}
\nu \ :=\ \min \{m\in \mathbb{N}\;:\;S_{m+i}\geq S_{m}\text{ for all }i\geq
0\}.
\end{equation*}%
Moreover, let $\iota \in \mathbb{N}$ be the first weak ascending ladder
epoch of $S$,
\begin{equation*}
\iota \ :=\ \min \{m\in \mathbb{N}\;:\;S_{m}\geq 0\}\ .
\end{equation*}%
We denote
\begin{equation*}
\widetilde{f}_{n}\ :=\ f_{\nu +n}\ \mbox{ and }\ \widetilde{S}_{n}\ :=\
S_{\nu +n}-S_{\nu },\quad n\in \mathbb{N}.\
\end{equation*}

The previous result allows us to rigorously express what we mean by \textit{%
living in a good environment} for the population process. The next lemma and
its proof are the same as Lemma 2.6 in \cite{afanasyev2005criticality} and
its proof. We thus do not provide it and refer the reader to \cite%
{afanasyev2005criticality}.

\begin{lemma}
\label{L_tanaka} (see Lemma 2.6 in \cite{afanasyev2005criticality}) Suppose
that $\iota <\infty $ $\mathbf{P}$--$a.s.$ Then $\nu <\infty $ $\,\mathbf{P}%
^{+}$--$a.s.$ and

\begin{enumerate}
\item $(f_{1},f_{2},\ldots )$ and $(\widetilde{f}_{1},\widetilde{f}%
_{2},\ldots )$ are identically distributed with respect to~$\mathbf{P}^{+}$;

\item $(\nu ,f_{1},\ldots ,f_{\nu })$ and $(\widetilde{f}_{1},\widetilde{f}%
_{2},\ldots )$ are independent with respect to $\mathbf{P}^{+}$;

\item $\mathbf{P}^+ \{\nu = k, S_{\nu} \in dx \} = \mathbf{P} \{\iota = k,
S_{\iota} \in dx \}$ for all $k \geq 1$.
\end{enumerate}
\end{lemma}

\section{Proof of Theorem \protect\ref{T_main}}

\label{sec_proof}

Thanks to the results we have collected in the previous sections, we are now
able to prove our main result. We have already demonstrated (Lemmas \ref%
{conditioning} and \ref{L_tanaka}) that we can divide the survival
probability of $\mathcal{Z}$ until time $n$ into two parts: the probability
for the process to survive until the time when the running infimum $L_{n}$
is attained for the first time, and the probability that the process $%
\mathcal{Z}$ survives in a "good" environment, i.e., in an environment with
a running infimum of $L$ null. We still have to prove that the population
indeed has a nonnegligible probability to survive in this good environment,
for large $n$. It is the content of the next result.

Let
\begin{equation*}
\eta _{k}\ :=\ \sum_{y=0}^{\infty }y(y-1)\;f_{k}\left[ y\right] \ \Big/\ %
\Big(\sum_{y=0}^{\infty }y\;f_{k}\left[ y\right] \Big)^{2}\,,\quad k\in
\mathbb{N}.
\end{equation*}

\begin{lemma}
\label{L_borel} Assume that Conditions (\ref{BasicCond}) and (\ref{Def_mu}) hold. If $p>q$, and Conditions A and C hold or if $p<q$ and Condition
B holds, then
\begin{equation*}
\sum_{k=0}^{\infty }\eta _{k+1}e^{-S_{k}}\ <\ \infty \quad \mathbf{P}^{+}%
\text{--}a.s.
\end{equation*}
\end{lemma}

\begin{proof}
Let us first assume that $p>q$, and Conditions A and C hold. Recall the
definition of the standardized truncated second moment of the environment in %
\eqref{defzetaa}. Following \cite{afanasyev2005criticality} Equation (2.24)
we have the following bound, for any $a \in \mathbb{N}$,
\begin{align*}
\sum_{k=0}^{\infty }\eta _{k+1}e^{-S_{k}} &\leq a \sum_{k=0}^{\infty
}e^{-S_{k}} +\sum_{k=0}^{\infty } \zeta_{k+1}(a)e^{-S_{k}}  \notag \\
& =: \mathcal{A}_a+\mathcal{B}_a.  \label{defAandA}
\end{align*}

The first step of the proof consists in bounding the two sums by using the
times $0:=\nu (0)<\nu (1)<\cdots $ of prospective minima of $\mathcal{S}$,
defined by
\begin{equation*}
\nu (j)\ :=\ \min \{m>\nu ({j-1})\;:\;S_{m+i}\geq S_{m}\text{ for all }i\geq
0\}\,,\quad j\in \mathbb{N}.  \label{renewal}
\end{equation*}

By definition,
\begin{equation*}
S_{k}\geq S_{\nu (j)},\ \mbox{ if }\,k\geq \nu (j).  \label{simple}
\end{equation*}%
Thus, we get
\begin{equation*}
\mathcal{A}_a \leq a \sum_{j=0}^{\infty } (\nu(j+1)-\nu(j)) e^{-S_{v(j)}},
\end{equation*}%
and
\begin{equation*}
\mathcal{B}_a \leq \sum_{j=0}^{\infty } \left(\sum_{k=\nu(j)+1}^{\nu(j+1)}
\zeta_k(a) \right) e^{-S_{v(j)}}.
\end{equation*}

Now we aim at bounding the variables $\nu (j)$. For the sake of readability,
let us introduce
\begin{equation*}
\nu _{j}=\nu (j)-\nu (j-1),\quad j\in \mathbb{N}.
\end{equation*}%
By Lemma \ref{L_tanaka}.$($\emph{1)} and \emph{(2)}, $\nu (j)$ is the sum of
$j$ nonnegative i.i.d. random variables, each having the distribution of $%
\nu =\nu (1)=\nu _{1}$. Lemma \ref{L_tanaka}.$($\emph{3)} and (\ref{Ber3})
imply for large $k$
\begin{eqnarray*}
\mathbf{P}^{+}\left( \nu >k\right) \  &=&\ \mathbf{P}\{\iota >k\}\ =\
\mathbf{P}\{M_{k}<0\}\  \\
&\leq &2L_{44}(k)= 2/ \Lambda(k).
\end{eqnarray*}
These estimates and Condition\textbf{\ C} imply
\begin{equation*}
\sum_{j=1}^{\infty }\mathbf{P}^{+}\left( \nu _{j}>g(j)\right) \leq
2\sum_{j=1}^{\infty } 1/\Lambda(g(j))<\infty .
\end{equation*}
Hence, by the Borel-Cantelli lemma there will be $\mathbf{P}^{+}$--$a.s.$
only a finite number of cases when $\nu _{j}>g(j)$. And as $g(i)=e^{o(i)},\
i\rightarrow \infty $, for any $\gamma >0$,
\begin{equation*}
\sum_{i=0}^{j}g(i)=o\left( e^{\gamma j}\right) ,\quad j\rightarrow \infty .
\end{equation*}%
Thus, there will be $\mathbf{P}^{+}$--$a.s.$ only a finite number of cases
when $\nu (j)>e^{\gamma j}$.

Now we would like to bound the term
\begin{equation*}
\sum_{k=\nu (j)+1}^{\nu (j+1)}\zeta _{k}(a)
\end{equation*}%
in order to show that the random variable $\mathcal{B}_{a}$ is almost surely
finite. The first step to obtain this bound is to use the inequality (2.25)
in \cite{afanasyev2005criticality}, that we now recall: for any $x\geq 0$,
\begin{equation*}
\mathbf{P}^{+}(\zeta _{k}(a)>x)\leq \mathbf{P}(\zeta _{1}(a)>x)+\mathbf{E}%
[U(X_{1});\zeta _{1}(a)>x]\mathbf{P}(L_{k-1}\geq 0).
\end{equation*}%
Applying it with $x=k^{\alpha /\gamma }$ (with $\alpha >0$ to be precised
later on) and using the Markov inequality as well as Condition\textbf{\ A}
yield for any $k\in \mathbb{N}$,
\begin{equation*}
\mathbf{P}^{+}(\zeta _{k}(a)>k^{\alpha /\gamma })\leq \frac{c}{k^{\alpha
\beta /\gamma }}+\frac{c}{k^{\alpha \beta /\gamma }}\mathbf{P}(L_{k-1}\geq
0)\leq \frac{c}{k^{\alpha \beta /\gamma }}+\frac{c}{k^{\alpha \beta /\gamma }%
}\frac{\hat{l}_{22}(k)}{k},
\end{equation*}%
where we applied \eqref{EstimMin2} and the value of $c$ can change from line
to line. The constants $\alpha $ and $\beta $ are fixed. However, we know
that $\gamma $ can be chosen as small as we want. In particular, we may
select it in such a way that $\alpha \beta /\gamma =2$. Applying again the
Borel-Cantelli lemma we deduce that there is $\mathbf{P}^{+}$--$a.s.$ only a
finite number of cases when $\zeta _{k}(a)>k^{\alpha /\gamma }$.

Combining this fact with the previous results we obtain that for $j$ large
enough and $k\in \lbrack \nu (j-1)+1,\nu (j)]$, $\mathbf{P}^{+}$--$a.s.$,
\begin{equation*}
\zeta _{k}(a)\leq k^{\alpha /\gamma }\leq \left( e^{\gamma j}\right)
^{\alpha /\gamma }=e^{\alpha j}\quad \text{and}\quad \nu _{j}\leq g(j).
\end{equation*}%
Hence for $j$ large enough, $\mathbf{P}^{+}$--$a.s.$,
\begin{equation*}
\sum_{k=\nu (j-1)+1}^{\nu (j)}\zeta _{k}(a)\leq e^{\alpha j}\nu _{j}\leq
e^{\alpha j}g(j).
\end{equation*}

The last part of the proof consists in estimating the $S_{\nu (j)}$ from
below. According to Lemma \ref{L_tanaka} \emph{(1)} and \emph{(2)}, the
random variable $S_{\nu (j)}$ is the sum of $j$ non-negative i.i.d. random
variables with positive mean. Thus, there exists a $\lambda >0$ such that
\begin{equation*}
S_{\nu (j)}\ \geq \ \lambda j\qquad \text{eventually}\quad \mathbf{P}^{+}%
\text{--}a.s.  \label{lln}
\end{equation*}

Choosing $\alpha <\lambda $ in the previous inequalities, we obtain

\begin{align*}
\sum_{k=0}^{\infty }\eta _{k+1}e^{-S_{k}}\leq \mathcal{A}_{a}+\mathcal{B}%
_{a}& \leq c\sum_{j=0}^{\infty }(a+e^{\alpha (j+1)})\nu _{j+1}e^{-S_{v(j)}} \\
& \leq c\sum_{j=0}^{\infty }e^{\alpha (j+1)}g(j+1)e^{-\lambda j}<\infty \qquad
\mathbf{P}^{+}\text{--}a.s.,
\end{align*}%
where the value of $c$ can change from line to line. It ends the proof for
the case $p>q$.

The proof for the case $p<q$ is the same as the proof of Lemma 2.7 in \cite%
{afanasyev2005criticality}. Indeed, even if the authors of the mentioned
paper assume $\rho \in (0,1)$, their proof remains valid when we take $\rho
=1$ as it is the case when $p<q$.
\end{proof}

In the following example, we illustrate the fact that \textbf{Condition\ C}
is not too strong and that we can find slowly varying functions $L(\cdot )$
satisfying this assumption. To this aim we choose $g(j)=e^{j^{1-\theta }}$
for a $\theta \in (0,1)$.

\begin{example}
\label{ex1} Let us consider a slowly varying function $L(\cdot )$ satisfying
\begin{equation*}
L(x)\sim \frac{c}{\log ^{m+1}x},\quad m>(p-q)/p,
\end{equation*}%
where $c$ is a positive constant. Then, as $x\rightarrow \infty $%
\begin{equation*}
l^{\ast }(x)=\int_{x}^{\infty }\frac{L(u)}{u}du\sim \frac{c}{m\log ^{m}x}
\end{equation*}
and, by \eqref{defan}
\begin{equation*}
a_{n}\sim \frac{n}{\log ^{m+1}n}, \quad n\rightarrow \infty .
\end{equation*}%
Therefore,
\begin{equation*}
\left\vert h_{n}\right\vert \sim \left\vert \left( p-q\right) nl^{\ast
}\left( a_{n}\right) \right\vert \sim c_{2}\frac{n}{\log ^{m}n},\quad
n\rightarrow \infty ,
\end{equation*}%
for a positive $c_{2}.$

Now from Lemma 7.3 in \cite{Berg17}, we know that
$$ \mathbf{P}(S_n\geq 0) \sim \frac{p}{p-q} \frac{L(|h_n|)}{l^\ast(|h_n|)}  $$
Hence, using the previous calculations we obtain
$$ \mathbf{P}(S_n\geq 0) \sim \frac{p m}{(p-q)\ln n}.  $$
As
$$ \frac{p m}{(p-q)} \sum_{k=2}^n \frac{1}{k\ln k} \sim \frac{p m}{(p-q)}\ln \ln n, \quad n \to \infty, $$
an application of Corollary 1.7.3 in \cite{bingham1989regular} yields, as $n\to\infty$
$$ \sum_{k=1}^\infty  \frac{\mathbf{P}(S_k \geq 0)}{k}\left(1-\frac{1}{n}\right)^k \sim \frac{p m}{(p-q)} \ln \ln n ,$$
and, in particular, for $\theta\in (0,1)$ as $j\to\infty$
$$ \sum_{k=1}^\infty  \frac{\mathbf{P}(S_k \geq 0)}{k}\left(1-\frac{1}{e^{j^{1-\theta }}}\right)^k \sim \frac{p m}{(p-q)} \ln \ln e^{j^{1-\theta }}
=\frac{p m}{(p-q)}(1-\theta)\ln j.$$
As a consequence, for any $\varepsilon>0$, there exist $j(\varepsilon)$ such that for $j \geq j(\varepsilon)$,
$$ \Lambda (e^{j^{1-\theta }}) \geq j^{(1-\varepsilon)(1-\theta)p m/(p-q)}. $$
As $p m/(p-q)>1$, we just have to choose $\varepsilon > 0$ and $\theta\in (0,1)$ such that $(1-\varepsilon)(1-\theta)p m/(p-q)>1$ to conclude that
\textbf{Condition C}
holds for $p>q.$
\end{example}

Introduce iterations of probability generating functions $\ \
f_{1}(.),f_{2}(.),...$ by setting
\begin{equation*}
f_{k,n}(s):=f_{k+1}(f_{k+2}(\ldots (f_{n}(s))\ldots ))
\end{equation*}%
for $0\leq k\leq n-1$, $0 \leq s \leq 1$, and letting $f_{n,n}(s):=s.$ By
definition,%
\begin{equation*}
\mathbf{P}\left( Z_{n}>0|\ f_{k+1},\ldots ,f_{n};Z_{k}=1\right) =1-f_{k,n}(0)
\end{equation*}%
and we have (see, for instance, formula (3.4) in \cite%
{afanasyev2005criticality})%
\begin{equation*}
1-f_{0,n}(0)\geq \left( e^{-S_{n}}+\sum_{k=1}^{n-1}\eta
_{k+1}e^{-S_{k}}\right) ^{-1}
\end{equation*}%
implying by Lemma \ref{L_borel}
\begin{equation}
1-f_{0,\infty }(0):=\lim_{n\rightarrow \infty }\left( 1-f_{0,n}(0)\right)
\geq \left( \sum_{k=1}^{\infty }\eta _{k+1}e^{-S_{k}}\right) ^{-1}\text{ }%
\mathbf{P}^{+}-\text{a.s.}  \label{Est_f2}
\end{equation}

We finally provide the proof of our main result.

\begin{proof}[Proof of Theorem \protect\ref{T_main}]
Let us begin with the case $p>q$. We write
\begin{eqnarray}
\mathbf{P}\left( Z_{n}>0\right)  &=&\sum_{k=0}^{n}\mathbf{P}\left(
Z_{n}>0;\tau _{n}=k\right)   \notag  \label{majprop3} \\
&=&\sum_{k=0}^{N}\mathbf{E}\left[ 1-f_{0,n}(0);\tau _{n}=k\right]
+\sum_{k=N+1}^{n\varepsilon }\mathbf{E}\left[ 1-f_{0,n}(0);\tau _{n}=k\right]
\notag \\
&&+\sum_{k=n\varepsilon +1}^{n}\mathbf{E}\left[ 1-f_{0,n}(0);\tau _{n}=k%
\right] ,
\end{eqnarray}%
for some $N\in \mathbb{N}$ to be precised later on, and a small positive $%
\varepsilon $. Let us first bound the second term in the right hand side of %
\eqref{majprop3}
\begin{eqnarray*}
\sum_{k=N+1}^{n\varepsilon }\mathbf{E}\left[ 1-f_{0,n}(0);\tau _{n}=k\right]
&\leq &\sum_{k=N+1}^{n\varepsilon }\mathbf{E}\left[ 1-f_{0,k}(0);\tau _{n}=k%
\right]  \\
&=&\sum_{k=N+1}^{n\varepsilon }\mathbf{E}\left[ 1-f_{0,k}(0);\tau _{k}=k%
\right] \mathbf{P}\left( L_{n-k}\geq 0\right)  \\
&\leq &\sum_{k=N+1}^{n\varepsilon }\mathbf{E}\left[ e^{S_{k}};\tau _{k}=k%
\right] \mathbf{P}\left( L_{n-k}\geq 0\right) .
\end{eqnarray*}%
By the duality principle for random walks and Lemma \ref{C_exponential},
with $x=y=0$, we have
\begin{equation*}
\mathbf{E}\left[ e^{S_{k}};\tau _{k}=k\right] =\mathbf{E}\left[
e^{S_{k}};M_{k}<0\right] \leq c\ b_{k}.
\end{equation*}%
This estimate, the equivalence
\begin{equation*}
b_{k}=\frac{1}{ka_{k}}\sim \frac{1}{k^{2}L_{4}(k)}
\end{equation*}%
and \eqref{Ber1} give%
\begin{eqnarray*}
\sum_{k=N+1}^{n\varepsilon }\mathbf{E}\left[ 1-f_{0,n}(0);\tau _{n}=k\right]
&\leq &\mathbf{P}\left( L_{n(1-\varepsilon )}\geq 0\right)
\sum_{k=N+1}^{\infty }\frac{1}{k^{2}L_{4}(k)} \\
&\leq &C\frac{\mathbf{P}\left( L_{n(1-\varepsilon )}\geq 0\right) }{a_{N}}.
\end{eqnarray*}%
Now we focus on the third part of the right hand side of \eqref{majprop3}.
Similarly as for the second part, we have the following series of
inequalities, where the value of the finite constant $C$ may change from
line to line and may depend on $\varepsilon $:
\begin{eqnarray}
\sum_{k=n\varepsilon +1}^{n}\mathbf{E}\left[ 1-f_{0,n}(0);\tau _{n}=k\right]
&\leq &C\sum_{k=n\varepsilon +1}^{n}\frac{1}{k^{2}L_{4}(k)}\mathbf{P}\left(
L_{n-k}\geq 0\right)   \notag  \label{maj_part3} \\
&\leq &C\frac{L_{22}(n)}{n^{2}L_{4}(n)}=o\left( \frac{1}{n^{3/2}}\right) .
\end{eqnarray}%
Finally,
\begin{equation*}
\sum_{k=0}^{N}\mathbf{E}\left[ 1-f_{0,n}(0);\tau _{n}=k\right]
=\sum_{k=0}^{N}\mathbf{E}\left[ 1-f_{k,n}^{Z_{k}}(0);\tau _{k}=k,L_{k,n}\geq
0\right]
\end{equation*}%
where
\begin{equation*}
L_{k,n}=\min_{k\leq j\leq n}\left( S_{j}-S_{k}\right) .
\end{equation*}%
Recalling Lemma \ref{conditioning} and using the independency and
homogeneity of the environmental components we conclude that for $k\leq N$,
\begin{eqnarray*}
\mathbf{E}\left[ 1-f_{k,n}^{Z_{k}}(0);\tau _{k}=k,L_{k,n}\geq 0\right]
&=&\sum_{j=1}^{\infty }\mathbf{P}\left( Z_{k}=j,\tau _{k}=k\right) \mathbf{P}%
\left( L_{n-k}\geq 0\right) \mathbf{E}\left[ 1-f_{0,n-k}^{j}(0)|L_{n-k}\geq 0%
\right]  \\
&\sim &\mathbf{P}\left( L_{n}\geq 0\right) \sum_{j=1}^{\infty }\mathbf{P}%
\left( Z_{k}=j,\tau _{k}=k\right) \mathbf{E}^{+}\left[ 1-f_{0,\infty }^{j}(0)%
\right]
\end{eqnarray*}%
as $n\rightarrow \infty $. Note that by Lemma \ref{L_borel}\ and (\ref%
{Est_f2})
\begin{equation*}
\mathbf{E}^{+}\left[ 1-f_{0,\infty }^{j}(0)\right] \geq \mathbf{E}^{+}\left[
1-f_{0,\infty }(0)\right] \geq \mathbf{E}^{+}\left[ \left(
\sum_{k=0}^{\infty }\eta _{k+1}e^{-S_{k}}\right) ^{-1}\right] >0.
\end{equation*}%
Thus, letting first $n$ to infinity, then $\varepsilon $ to zero and,
finally, $N$ to infinity we prove (\ref{survivP>Q}).

The proof for the case $q>p$ is very similar. The only difference is when
looking for an equivalent of Equation \eqref{maj_part3}. Applying %
\eqref{Ber2} yields
\begin{eqnarray*}
\sum_{k=n\varepsilon +1}^{n}\mathbf{E}\left[ 1-f_{0,n}(0);\tau _{n}=k\right]
&\leq &C\sum_{k=n\varepsilon +1}^{n}\frac{1}{k^{2}L_{4}(k)}\mathbf{P}\left(
L_{n-k}\geq 0\right) \\
&\leq &C_{\varepsilon }\frac{L_{33}(n)}{nL_{4}(n)}=o\left( \frac{1}{n^{1/2}}%
\right) =o\left( \mathbf{P}\left( L_{n}\geq 0\right) \right) .
\end{eqnarray*}%
We end the proof as for \eqref{survivP>Q}.
\end{proof}

\begin{acknow}
The authors thank I. Kortchemski for attracting their attention to reference \cite{kortchemski2019condensation} which
allowed them to weaken the conditions for their main result to hold.
C.S. thanks the CNRS for its financial support through its competitive
funding programs on interdisciplinary research, the ANR ABIM 16-CE40-0001 as
well as the Chair "Mod\'elisation Math\'ematique et Biodiversit\'e" of
VEOLIA-Ecole Polytechnique-MNHN-F.X.
She also thanks the French Embassy in Russia for its financial support via the Andr\'e Mazon program.
\end{acknow}

\end{document}